\UseRawInputEncoding
\documentclass[12pt]{article}

\usepackage{amsmath,amssymb,amsthm}
\usepackage[margin=1in]{geometry}
\usepackage{booktabs}
\usepackage{hyperref}

\newtheorem{theorem}{Theorem}
\newtheorem{proposition}[theorem]{Proposition}
\newtheorem{corollary}[theorem]{Corollary}
\newtheorem{conjecture}[theorem]{Conjecture}
\theoremstyle{definition}
\newtheorem{definition}[theorem]{Definition}
\newtheorem{remark}[theorem]{Remark}
\newtheorem{observation}[theorem]{Observation}

\newcommand{\BSR}{\mathrm{BSR}}
\newcommand{\Z}{Z}
\newcommand{\x}{{\bf x}}
\newcommand{\y}{{\bf y}}

\title{Overtake and the recursive-line equality
for second-order Zarankiewicz numbers}
\author{Chen Ouyang\footnote{School of Computer Science and Technology,
		Dongguan University of Technology, Dongguan 523000, China. ({\tt oych26@163.com}).}
		\and Liqun Qi\footnote{Jiangsu Provincial Scientific Research Center of Applied Mathematics, Nanjing 211189, China.
			Department of Applied Mathematics, The Hong Kong Polytechnic University, Hung Hom, Kowloon, Hong Kong.
			({\tt maqilq@polyu.edu.hk})}
	}
\date{\today}

\begin{document}

\maketitle

\begin{abstract}
We study two questions about the second-order Zarankiewicz numbers.
The first is the recursive-line equality
$z_2(m,n)=z_{RL}(m,n)$, known to hold in every computed case and
conjectured in general.  The second, here called \emph{overtake}, is
the strict inequality $z_{2A}(m,n)>z_2(m,n)$ between the
second-order augmented number and the second-order number.  Both are
governed by the number $H_{\mathrm{limited}}(m,n)$ of unoccupied
cells in a limited extremal configuration.  We prove that the first
increase $z_2\to z_2+1$ already costs three unoccupied cells,
forcing a unit of one-edge deficit, and the cell identity then bounds
the excess $z_{2A}-z_2$ by
$\lfloor(H_{\mathrm{limited}}-1)/2\rfloor$.  In particular, when
the recursive-line hole count is at most two, no overtake occurs.
We conjecture the sharp budget $\lfloor H_{\mathrm{limited}}/3\rfloor$
and that $H_{\mathrm{limited}}\ge5$ is the threshold for an
overtake.  In the computed range $n\le7$, no overtake occurs at any
order, the sole exception being the known case $(4,4)$.  The recursive-line equality $z_2=z_{RL}$ is verified
at every order in the range where both values are determined; the
orders at which it remains unsettled are $(12,4)$ and $(14,4)$ at
width four, $(11,5)$ at width five, and $(8,7)$ at width seven.
\end{abstract}



\section{Introduction}

Denote \([n]=\{1,\dots,n\}\).  The classical Zarankiewicz number
\(z(m,n)\) is defined as the maximum number of edges of a
\(C_4\)-free \(m\times n\) bipartite graph \(G=(S,T,E)\), where the
vertex sets \(S\) and \(T\) have \(m\) and \(n\) vertices, and each
edge has the form \((i,j)\) with \(i\in S\), \(j\in T\); see
\cite{za51,kst54,culik56,re58,gu69,ni10,chm24,bo04}.  It was found
in \cite{qcx26a} that \(z(m,n)\) is equal to the maximum SOS rank of
the \(m\times n\) SOS simple biquadratic form
\[
f(\x,\y)=\sum_{(i,j)\in E} x_i^2 y_j^2 .
\]

Recently, L\"ofberg and Qi \cite{lq26} introduced the
\emph{second-order Zarankiewicz number} \(z_2(m,n)\) as the maximum
SOS rank of the \(m\times n\) SOS double simple biquadratic form
\[
P_G(\x,\y)=\sum_{(i,j)\in E_1} x_i^2 y_j^2
+\sum_{(i,j;k,l)\in E_2} (x_i y_j + x_k y_l)^2 ,
\]
{where \(E_1\)  and \(E_2\)  are,  respectively,  the 1-edge set  and the 2-edge set of  an augmented bipartite graph \(G=(S,T,E_1\cup E_2)\)  with  \(|E_1|=z(m,n)\)  and no two edges occupy the same cell (i.e., a pair \((i,j) \) with \(i\in S ,~j\in T\)  can only appear in one edge); the form \(P_G\) is called the \emph{displayed form} of \(G\).  Its \emph{SOS rank} is the least number of squares of real bilinear forms in \((\x,\y)\) whose sum is \(P_G\), so that \(P_G\) is irreducible exactly when this rank equals the displayed length \(R(G)=|E_1|+|E_2|\); the maximum defining \(z_2(m,n)\) is taken over irreducible displayed forms. }
They also
introduced the \emph{recursive-line Zarankiewicz number}
\(z_{RL}(m,n)\), which is a checkable parameter, together with the
\emph{signed Zarankiewicz number} \(z_{SL}(m,n)\), giving the
unconditional hierarchy
\[
\mathrm{BSR}(m,n)\ \ge\ z_2(m,n)\ \ge\ z_{SL}(m,n)
\ \ge\ z_{RL}(m,n)\ \ge\ z(m,n),
\]
where \(\mathrm{BSR}(m,n)\) is the maximum SOS rank of all
\(m\times n\) SOS biquadratic forms
\[
f(\x,\y)=\sum_{i,k\in[m],\,j,l\in[n]} a_{ijkl}\,x_i y_j x_k y_l .
\]
The recursive-line and signed numbers are maxima of the same
displayed length \(R(G)=|E_1|+|E_2|\) over the limited
configurations whose closure satisfies the criterion
\((\mathrm{RW}3^+)\), respectively, the signed criterion
\((\mathrm{RW}3^\pm)\).

\smallskip\noindent\textbf{Notation.}  For an augmented bipartite
graph \(G=(S,T,E_1\cup E_2)\), a cell \((i,j)\) with \(i\in S\),
\(j\in T\) is \emph{occupied} by \(G\) if it is a 1-edge of \(E_1\) or
one half of a 2-edge of \(E_2\), and \(H(G)\) denotes the number of
\emph{unoccupied} cells.  The \emph{displayed length}
\(R(G)=|E_1|+|E_2|\) of \(G\) is the number of displayed squares in
\(P_G\).  A simple configuration \(G\) whose one-edge graph is
\(C_4\)-free with \(|E_1|=z(m,n)\) is called \emph{limited}, and
\emph{limited extremal} when in addition it is irreducible and
\(R(G)=z_2(m,n)\).  We write \(H_{\mathrm{limited}}(m,n)\) for the
number of unoccupied cells in a limited extremal configuration and
\(d(G)=z(m,n)-|E_1(G)|\) for the \emph{one-edge deficit} of a
configuration \(G\), and
\(h_{\mathrm{RL}}(m,n)=mn+z(m,n)-2z_{RL}(m,n)\) for the \emph{recursive-line hole
count}.
The precise definitions, and the notion of irreducibility, are given in
Section~\ref{sec:defs}.

Two questions about the relation among the three parameters
\(z_{RL}\le z_{SL}\le z_2\) motivate this paper.

\smallskip\noindent\textbf{The recursive-line equality.}  It was
discovered in \cite{lq26,xy26} that
\begin{equation}\label{e1}{
z_2(m,n)=z_{SL}(m,n)=z_{RL}(m,n)
\le \Z(m,n):=\left\lfloor\frac{mn+z(m,n)}{2}\right\rfloor}
\end{equation}
in every case that has been computed.  Here \(\Z(m,n)\) is the
\emph{universal cell bound} of \cite{lq26}, and when equality holds throughout
\eqref{e1}, that is when \(z_2(m,n)=z_{SL}(m,n)=z_{RL}(m,n)=\Z(m,n)\),
either all cells are occupied or exactly one is unoccupied.  The identity {\(z_{RL}=z_{SL}= z_2\)} is established in many cases
\cite{lq26,cc26,ccq26,clq26,leb26a,xc26}, and it is conjectured to
hold for every \(m,n\ge2\).  By the hierarchy
\(z_{RL}\le z_{SL}\le z_2\), this identity is equivalent to the single
equation \(z_2=z_{RL}\), and it implies each of the weaker assertions
\(z_2=z_{SL}\) and \(z_{SL}=z_{RL}\), neither of which is known in
general, the second being the question of \cite{ccq26}.

\smallskip\noindent\textbf{Overtake.}  Very recently, Qi and Cui
\cite{qc26} introduced the \emph{second-order augmented
Zarankiewicz number} \(z_{2A}(m,n)\) by dropping the limitedness
constraint \(|E_1|=z(m,n)\) while keeping the irreducibility of the
displayed form \(P_G\).  This gives the expanded hierarchy
\[
\mathrm{BSR}(m,n)\ \ge\ z_{2A}(m,n)\ \ge\ z_2(m,n)
\ \ge\ z_{SL}(m,n)\ \ge\ z_{RL}(m,n)\ \ge\ z(m,n).
\]
Qi and Cui exhibited the first separation
\[
z_{2A}(4,4)=11>10=z_2(4,4),
\]
so the augmented parameter is a strictly better lower bound of
\(\mathrm{BSR}\) than the second-order number at \((4,4)\).  We call
the phenomenon \(z_{2A}(m,n)>z_2(m,n)\) an \emph{overtake}.

{Both questions are studied via the quantity \(H_{\mathrm{limited}}(m,n)\),  the number of unoccupied cells in a \emph{limited extremal configuration}. }

The main results of this paper are the following.

\begin{itemize}
\item \emph{Hole budget (Theorem~\ref{thm:budget}).}  The identity
\(2(R(G)-z_2)=H_{\mathrm{limited}}-H(G)-d(G)\) holds, and an overtake
forces \(H_{\mathrm{limited}}(m,n)\ge3\), with the budget
\(z_{2A}\le z_2+\lfloor(H_{\mathrm{limited}}-1)/2\rfloor\).

\item \emph{Low-hole reduction (Theorem~\ref{thm:low-hole}).}  If
\(h_{\mathrm{RL}}(m,n)\le2\), then
\(z_{2A}(m,n)=z_2(m,n)\), i.e., no overtake occurs.

\item \emph{Sharp budget and threshold (Conjectures~\ref{conj:budget}
and~\ref{conj:threshold}).}  We conjecture the sharp budget
\(z_{2A}\le z_2+\lfloor H_{\mathrm{limited}}/3\rfloor\), and that an
overtake occurs if and only if \(H_{\mathrm{limited}}\ge5\), with
\(\min(m,n)\ge3\) needed because the width-two grid \(2\times6\) is
the degenerate exception.

\item \emph{Computed range \(n\le7\) (Theorems~\ref{thm:n4-cases},
\ref{thm:n5-cases}, \ref{thm:n6-cases} and~\ref{thm:n7-cases}).}
At every order of widths \(4,5,6,7\) with
\((m,n)\ne(4,4)\) at which all three values \(z_{2A}\), \(z_{RL}\),
\(z_2\) are determined, one has
\(z_{2A}(m,n)=z_{RL}(m,n)=z_2(m,n)\).  The orders where the
recursive-line equality remains unsettled are \((12,4)\), where
\(z_{RL}=32\) and \(z_2\in[32,33]\), \((14,4)\), where
\(z_{RL}\in[37,38]\), \((11,5)\), where \(z_{RL}=37\) and
\(z_2\in[37,38]\), and \((8,7)\), where \(z_{RL}=38\) and
\(z_2\in[38,39]\); no overtake occurs at any of these four orders.
  The widths two
and three are settled in Section~2.

\item \emph{Almost-everywhere cell bound
(Proposition~\ref{prop:ae}).}  Conditionally on
\(H_{\mathrm{limited}}\le4\) off \((4,4)\) and off the degenerate
widths \(\min(m,n)\le2\), one has
\(z_{2A}(m,n)\ge\Z(m,n)-2\) for every \((m,n)\) with
\(\min(m,n)\ge3\), the value at \((4,4)\) being
\(z_{2A}(4,4)=11=\Z(4,4)-1\).
\end{itemize}

These results are distinct from two phenomena already present in
the literature.  The first is the separation \(z_2>z_L\), where
\(z_L\) is the first-order limited augmented Zarankiewicz number of
\cite{qcx26a,xyq26}: this separation is common and does not by
itself produce an overtake.  The second is the saturation
\(z_2=\Z\) at \((15,6)\) and \((21,7)\), proved in \cite{lq26},
which forces \(z_{2A}=z_2\) there and shows that the augmented
parameter never improves on the second-order one at those orders.
More generally, Lebedev's fixed-width deficit theorem
\cite[Theorem~1.2]{leb26a} bounds the gap
\(\Z-z_{RL}\), and with it the augmented excess
\(z_{2A}-z_2\le\Z-z_2\le\Z-z_{RL}\), by \(n^3/6+O(n^2)\) at each
fixed width; being a fixed-width statement, it does not by itself
control the case when both dimensions grow.  That case is covered by
the uniform half-density theorem and its unrestricted comparison
corollary \cite[Theorem~1.1 and Corollary~1.3]{leb26a}.  The two
regimes together give \(\Z-z_{RL}=o(mn)\) as
\(\max(m,n)\to\infty\), so any overtake is invisible at leading
order and
\(z_{2A}=z_2+o(mn)=z_{RL}+o(mn)\) uniformly as both dimensions grow.

The paper is organized as follows.  Section~2 provides the notation, the cell identity and the hole-budget theorem, establishes the low-hole reduction, states the sharp-budget and threshold conjectures, records the conditional almost-everywhere cell bound,
settles the two degenerate widths \(n=2\) and \(n=3\), and discusses the two questions in the light of Lebedev's fixed-width deficit theorem.  Sections~3 to~6 treat the widths \(n=4,5,6,7\) in turn:
for each width we collect the known values of \(z_{RL}\) and
\(H_{\mathrm{limited}}\), apply the low-hole reduction at every
order with \(h_{\mathrm{RL}}\le2\), reduce the remaining overtake
question to a finite enumeration of candidate configurations, and conclude that no overtake occurs at any determined order except \((4,4)\).  The recursive-line equality \(z_2=z_{RL}\) is thereby established wherever the computed range determines it, the open orders being the four listed above.  Section~\ref{sec:discussion} closes the paper by separating the overtake and recursive-line equality questions and explaining the role of the hole count \(h_{\mathrm{RL}}\) in each.

\section{Definitions and Basic Theory}\label{sec:defs}
{
A \(m\times n\) bipartite graph \(G_0=(S,T,E_1)\) is a graph having vertex partition \(S\cup T \) with \(|S|=m, |T|=n\), such that each edge \((i,j)\in E_1\) satisfies \(i\in S\) and \(j\in T\).
\begin{definition}[Augmented bipartite graphs]
An \(m \times n\) augmented bipartite graph \(G=(S,T,E_1\cup E_2)\)  is constructed from an \(m \times n\)  bipartite graph \(G_0=(S,T,E_1)\) by adding a 2-edge set \(E_2\). A 2-edge has the form \((i,j;k,l)\), where  vertices  \(i,k\)  are in \(S\) and \(j,l\)  in \(T\); its two cells \((i,j)\) and \((k,l)\) are distinct, and they may share a row or a column.
\end{definition}
 A cell \((i,j)\) is called occupied in \(G\) if  it is a 1-edge in \(E_1\) or  half of a 2-edge in \(E_2\).  Denote by \(R(G)=|E_1|+|E_2|\)  the displayed length (i.e., the number of displayed squares) of
 \[P_G(x,y)=\sum_{(i,j)\in E_1} x_i^2 y_j^2
+\sum_{(i,j;k,l)\in E_2} (x_i y_j + x_k y_l)^2 \] associated with \(G\)   and let \(H(G)\) be the number of unoccupied cells.  \begin{definition}[SOS rank]\label{def:sosrank}
The \emph{SOS rank} \(\operatorname{SOS}(P)\) of a biquadratic form \(P\) is the least number of squares of real bilinear forms in \((\x,\y)\) whose sum is \(P\); equivalently, writing \(v(\x,\y)=(x_iy_j)_{(i,j)\in[m]\times[n]}\), it is the least rank of a positive semidefinite matrix \(Q\) with \(v^{\mathsf T}Qv=P\).
\end{definition}
The configuration \(G\) is called \emph{irreducible} if  \(\operatorname{SOS}(P_G(x,y))=R(G)\).
We call that \(G\) satisfies \emph{simplicity} if no two edges in \(G\)  occupy a common cell , and then it yields the \emph{cell identity}
\begin{equation}\label{eq:cell}
2R(G)=mn+|E_1|-H(G),
\end{equation}
and further leads to}
\begin{equation}\label{eq:Z}
z_2(m,n)\le z_{2A}(m,n)\le \Z(m,n):=
\left\lfloor\frac{mn+z(m,n)}{2}\right\rfloor .
\end{equation}

\begin{definition}[Overtake and recursive-line equality]
An \emph{overtake} at \((m,n)\) is a strict inequality
\(z_{2A}(m,n)>z_2(m,n)\).  The \emph{recursive-line equality} is
\(z_2(m,n)=z_{RL}(m,n)\).  The \emph{recursive-line augmented number}
\(z_{RLA}(m,n)\) is defined from the recursive-line criterion
\((\mathrm{RW}3^+)\) by dropping the limitedness constraint, exactly as
\(z_{2A}\) is defined from \(z_2\); the \emph{recursive-line overtake} is
the analogous inequality \(z_{RLA}(m,n)>z_{RL}(m,n)\).
\end{definition}
{
\begin{definition}[Limited extremal configuration]
If an \(m\times n\) augmented bipartite graph \(G\) satisfies simplicity and the bipartite graph \(G_0\) is \(C_4\)-free with \(|E_1|=z(m,n)\),  then it is called a limited (augmented) configuration. Furthermore, if \(G\) is irreducible and \(R(G)=z_2(m,n)\), then it is called a \emph{limited extremal configuration}.
\end{definition}

\begin{remark}[How reducibility is certified]\label{rem:cert}
The Gram matrix of the displayed representation of \(P_G\) has rank exactly \(R(G)\), because the displayed bilinear forms have pairwise disjoint supports.  Reducibility of \(P_G\) is therefore certified not by the rank of that matrix but by exhibiting a second feasible Gram matrix, that is a positive semidefinite \(Q\) with \(v^{\mathsf T}Qv=P_G\) and \(\operatorname{rank}Q\le R(G)-1\), equivalently a representation of \(P_G\) by at most \(R(G)-1\) squares of bilinear forms.  Rank minimization over this feasible set is what the finite analyses of Sections~\ref{sec:n=4}--\ref{sec:n=7} carry out.
\end{remark}

Let \(H_{\mathrm{limited}}(m,n)\) denote the number of unoccupied
cells in a limited extremal configuration, then by  \eqref{eq:cell} }
\[
2z_2(m,n)=mn+z(m,n)-H_{\mathrm{limited}}(m,n).
\]
Hence \(H_{\mathrm{limited}}\) is the deficit of \(z_2\) from the
universal cell bound \(\Z\).  Since \(z_2\ge z_{RL}\), one has
\(H_{\mathrm{limited}}(m,n)\le h_{\mathrm{RL}}(m,n)\), where
\[h_{\mathrm{RL}}(m,n)=mn+z(m,n)-2z_{RL}(m,n).\]

\subsection{The \((4,4)\) overtake}

The paper \cite{qc26} proves the following.

\begin{theorem}[{\cite[Theorem 3.1]{qc26}}]\label{thm:44}
\[
\BSR(4,4)\ge z_{2A}(4,4)=11>10=z_2(4,4)=z_{RL}(4,4).
\]
The witness is a non-limited configuration with
\(|E_1|=8\), \(|E_2|=3\), \(R(G)=11\), \(H(G)=2\), while the
limited extremal configuration at \((4,4)\) has
\(H_{\mathrm{limited}}(4,4)=5\).  At \((4,4)\) the recursive-line
criterion \((\mathrm{RW}3^+)\) also certifies the same witness, thus
\(z_{RLA}(4,4)=z_{2A}(4,4)=11>10=z_{RL}(4,4)=z_2(4,4)\).
\end{theorem}

The upper bound \(z_{2A}(4,4)\le11\) is obtained from the cell bound
\(\Z(4,4)=12\) by excluding \(R(G)=12\): the identity \eqref{eq:cell}
forces \((|E_1|,H)=(9,1)\) or \((8,0)\); the first is limited and
would give \(z_2(4,4)\ge12\), the second is disposed of by a finite
classification of the \(161\) orbits \cite{qc26}.

\subsection{The computed range and saturation}

In the range computed for \(n\le7\), the values are as follows.  Here
\(z_L\) is the first-order limited augmented Zarankiewicz number of
\cite{qcx26a,xyq26}, \(z_2\) the second-order number, \(z_{2A}\) the
second-order augmented number, and \(\Z\) the universal cell bound.

\begin{table}[h]
\centering
\begin{tabular}{ccccc}
\toprule
\((m,n)\) & \(z_L\) & \(z_2\) & \(z_{2A}\) & overtake?\\
\midrule
\((4,3)\) & \(8\) & \(8\) & \(8\) & no\\
\((4,4)\) & --- & \(10\) & \(11\) & \textbf{yes}\\
\((5,5)\) & \(15\) & \(17\) & \(17\) & no\\
\((6,4)\) & \(14\) & \(16\) & \(16\) & no\\
\((6,5)\) & --- & \(20\) & \(20\) & no\\
\((15,6)\) & --- & \(\Z\) & \(\Z\) & no\\
\((21,7)\) & --- & \(\Z\) & \(\Z\) & no\\
\bottomrule
\end{tabular}
\caption{Selected values in the computed range.  The only second-order
overtake is \((4,4)\).  At \((5,5)\) and \((6,4)\) the second-order
number is \(z_2(5,5)=17\) and \(z_2(6,4)=16\) by \cite{xyq26}, with
\(z_{2A}=z_2\); the first-order limited augmented number is strictly
smaller, a separation known from \cite{xyq26}.  At \((15,6)\) and
\((21,7)\) the cell bound is saturated by \cite{lq26}, so
\(z_{2A}=z_2\).  At \((6,5)\) the six-column value
\(z_2(6,5)=20\) of \cite[Theorem~1.4]{leb26a} (transposed) is
recovered in Section~\ref{sec:n=5}, where
Proposition~\ref{prop:n5-cases} gives
\(z_{2A}(6,5)=z_{RL}(6,5)=z_2(6,5)=20\).  This table is not exhaustive;
complete width-by-width data are given in
Sections~\ref{sec:defs}--\ref{sec:n=7}.}
\label{tab:range}
\end{table}

\begin{proposition}[Saturation implies no overtake]\label{prop:sat}
If \(z_2(m,n)=\Z(m,n)\), then \(z_{2A}(m,n)=z_2(m,n)\).
In particular, at \((15,6)\) and \((21,7)\), where \(z_2=\Z\) by
\cite{lq26}, no overtake occurs.
\end{proposition}

\begin{proof}
The hierarchy \eqref{eq:Z} gives \(z_2\le z_{2A}\le\Z\).  If
\(z_2=\Z\), then both inequalities are equalities, so \(z_{2A}=z_2\).
\end{proof}

Thus, in the computed range,
\(z_{2A}(m,n)=z_2(m,n)\) at every order other than \((4,4)\), where
the overtake occurs. 

\subsection{The hole-budget theorem}

\begin{observation}\label{obs:identity}
Let \(G\) be a limited extremal configuration with \(|E_1|=z(m,n)\)
and \(H(G)=H_{\mathrm{limited}}(m,n)\).  If \(G'\) is any simple
\(C_4\)-free configuration with \(R(G')=R(G)+1\), then, writing
\(d=z(m,n)-|E_1'|\) and \(H'=H(G')\), the cell identity gives
\begin{equation}\label{eq:deficit}
H'=H_{\mathrm{limited}}-2-d.
\end{equation}
\end{observation}

\begin{proof}
Subtract \(2R(G)=mn+z(m,n)-H_{\mathrm{limited}}\) from
\(2(R(G)+1)=mn+|E_1'|-H'\) and rearrange.
\end{proof}

\begin{theorem}[Hole budget of an overtake]\label{thm:budget}
Let \(G\) be a simple configuration with \(C_4\)-free one-edge graph
on the \(m\times n\) grid, and write \(d(G)=z(m,n)-|E_1(G)|\) for
its one-edge deficit.  Then
\begin{equation}\label{eq:budget}
2\bigl(R(G)-z_2(m,n)\bigr)=H_{\mathrm{limited}}(m,n)-H(G)-d(G).
\end{equation}
Consequently, for \(H_{\mathrm{limited}}(m,n)\ge1\),
\begin{equation}\label{eq:half}
z_2(m,n)\le z_{2A}(m,n)\le z_2(m,n)
+\left\lfloor\frac{H_{\mathrm{limited}}(m,n)-1}{2}\right\rfloor ,
\end{equation}
and every \emph{irreducible} configuration with \(R(G)>z_2(m,n)\)
is non-limited, has \(d(G)\ge1\), and satisfies
\begin{equation}\label{eq:three}
H(G)\le H_{\mathrm{limited}}(m,n)-3 .
\end{equation}
In particular an overtake forces \(H_{\mathrm{limited}}(m,n)\ge3\).
\end{theorem}

\begin{proof}
Apply the cell identity \eqref{eq:cell} to \(G\) and to a limited
extremal configuration { \(G_L\) , for which
\(|E_1(G_L)|=z(m,n)\), \(H(G_L)=H_{\mathrm{limited}}(m,n)\) and
\(2z_2(m,n)=mn+z(m,n)-H_{\mathrm{limited}}(m,n)\).  Subtracting the
second identity from the first gives \eqref{eq:budget}, since
\(|E_1|-z(m,n)=-d(G)\).
If \(R(G)>z_2(m,n)\), then the right-hand side of
\eqref{eq:budget} is a positive even integer.  If in addition \(G\)
is irreducible, then \(G\) cannot be limited, since a limited
irreducible configuration has \(R(G)\le z_2(m,n)\) by the definition
of \(z_2\); hence \(d(G)\ge1\), and we obtain by \eqref{eq:budget}
that
\[H_{\mathrm{limited}}-H(G)=2(R(G)-z_2)+d(G)\ge2+d(G)\ge3,\] which is
\eqref{eq:three} and thus \(H_{\mathrm{limited}}\ge3\).
In particular,
if \(R(G)=z_{2A}>z_2,\) then (such a \(G\) is irreducible, being
counted by \(z_{2A}\))
\(H(G)+d(G)\ge 1\) , and by   \eqref{eq:budget} we have
\[2(z_{2A}-z_2)=H_{\mathrm{limited}}-H(G)-d(G)\le H_{\mathrm{limited}}-1.\]
 Hence}
\(z_{2A}-z_2\le\lfloor(H_{\mathrm{limited}}-1)/2\rfloor\), which is
\eqref{eq:half}.

Irreducibility cannot be dropped from the statement: on the
\(4\times4\) grid a limited configuration with nine \(1\)-edges,
three \(2\)-edges and one hole has \(R(G)=12>z_2(4,4)=10\) and
\(d(G)=0\), and it is necessarily reducible, since irreducibility
would give \(z_2(4,4)\ge12\).
\end{proof}

\begin{theorem}[Low-hole reduction]\label{thm:low-hole}
Let \(h_{\mathrm{RL}}(m,n)=mn+z(m,n)-2z_{RL}(m,n)\).  If
\(h_{\mathrm{RL}}(m,n)\le2\), then \(z_{2A}(m,n)=z_2(m,n)\).
\end{theorem}

\begin{proof}
Let \(G\) be a configuration counted by \(z_{2A}(m,n)\), so \(G\) is
irreducible and \(R(G)=z_{2A}(m,n)\).  The cell identity gives
\[
2\bigl(z_{2A}(m,n)-z_{RL}(m,n)\bigr)
=h_{\mathrm{RL}}(m,n)-H(G)-d(G),
\]
as in \eqref{eq:budget} with \(z_2\) replaced by \(z_{RL}\).  Suppose
\(z_{2A}(m,n)>z_2(m,n)\).  The left-hand side is then a positive even
integer, so \(h_{\mathrm{RL}}-H-d\ge2\).  Since \(H,d\ge0\) and
\(h_{\mathrm{RL}}\le2\), this forces \(h_{\mathrm{RL}}=2\) and
\(H(G)=d(G)=0\).  In particular \(G\) is limited (\(d(G)=0\)) and
irreducible (as it is counted by \(z_{2A}\)), so \(R(G)\le z_2(m,n)\),
contradicting \(R(G)=z_{2A}(m,n)>z_2(m,n)\).  Hence
\(z_{2A}(m,n)=z_2(m,n)\).
\end{proof}

\begin{remark}[The unit move and its cost]\label{rem:cost}
A change of configuration that raises the displayed length by one
while keeping simplicity is constrained by the cell identity only: it
satisfies \(\Delta|E_1|+\Delta|E_2|=1\) and
\(\Delta H=\Delta|E_1|-2\), that is
\[
(\Delta|E_1|,\Delta|E_2|,\Delta H)=(t,1-t,t-2),\qquad t\in\mathbb Z ,
\]
and simplicity does not single out the value \(t=-1\).  That value is
the move which increases the one-edge deficit: a \(1\)-edge is deleted,
freeing one cell, and two \(2\)-edges are added, occupying four.  Its
cost is three holes, and it is the move that produces the \((4,4)\)
overtake.  Iterating it at three holes per unit of excess is the
statement of Conjecture~\ref{conj:budget}.
\end{remark}

\begin{corollary}\label{cor:budget}
Let \(\mu(m,n)=\min\{H(G)+d(G)\}\) over the configurations counted
by \(z_{2A}(m,n)\).  Then
\[
z_{2A}(m,n)\le z_2(m,n)
+\left\lfloor\frac{H_{\mathrm{limited}}(m,n)}{3}\right\rfloor\]
\[\iff
\mu(m,n)\ge H_{\mathrm{limited}}(m,n)
-2\left\lfloor\frac{H_{\mathrm{limited}}(m,n)}{3}\right\rfloor ,
\]
and, by the parity constraint \(\mu\equiv H_{\mathrm{limited}}\pmod2\),
this is equivalent to \(\mu(m,n)\ge\lceil H_{\mathrm{limited}}(m,n)/3\rceil\).
The inequality on the left holds whenever
\(H_{\mathrm{limited}}(m,n)\in\{0,1,2,3,4,6\}\).  For
\(H_{\mathrm{limited}}(m,n)\in\{5,7\}\) it is equivalent to the
non-attainment of the universal cell bound, i.e.,
\(z_{2A}(m,n)<\Z(m,n)\).  At \((4,4)\) this is the finite exclusion
of the value \(12\) in \cite[Theorem~3.1]{qc26}.
\end{corollary}

\begin{proof}
Write \(H=H_{\mathrm{limited}}(m,n)\) and recall that, by the budget
identity \eqref{eq:budget},
\[
2(z_{2A}(m,n)-z_2(m,n))=H-\mu(m,n),
\]
so the left side is even and \(\mu\equiv H\pmod 2\).

\emph{Case \(H=0\).}  Then \(z_2(m,n)=\Z(m,n)\), so
Proposition~\ref{prop:sat} gives \(z_{2A}=z_2\), and the sharp bound
and the assertion \(\mu\ge\lceil H/3\rceil=0\) hold trivially.

\emph{Case \(H\ge 1\).}  We first show \(\mu\ge 1\).  Suppose
\(\mu=0\), so there is a configuration \(G\) counted by
\(z_{2A}(m,n)\) with \(H(G)=d(G)=0\).  Such a \(G\) is limited (as
\(d(G)=0\)) and hole-free.  The budget identity gives
\(2(z_{2A}-z_2)=H\ge 1\), so \(z_{2A}>z_2\); but a limited
irreducible configuration has displayed length at most \(z_2\),
contradiction.  Hence \(\mu\ge 1\).

Given \(\mu\ge 1\) and \(\mu\equiv H\pmod 2\), the smallest
parity-admissible value of \(\mu\) is \(1\) if \(H\) is odd and
\(2\) if \(H\) is even.

The sharp bound \(z_{2A}\le z_2+\lfloor H/3\rfloor\) is equivalent,
via the budget identity, to
\[
\mu\ge H-2\lfloor H/3\rfloor .
\]
Writing \(H=3q+r\) with \(r\in\{0,1,2\}\), this reads \(\mu\ge q+r\).
Under the parity constraint \(\mu\equiv H\pmod 2\), the conditions
\(\mu\ge q+r\) and \(\mu\ge\lceil H/3\rceil=q+\lceil r/2\rceil\) are
equivalent: for \(r=0,1\) they coincide; for \(r=2\), the parity
constraint forces \(\mu\ge q+2\) and \(\mu\ge q+1\) to have the same
content.

Consequently the sharp bound holds whenever the smallest
parity-admissible value of \(\mu\) already satisfies the condition:
for \(H\in\{1,2,3,4,6\}\) this is the case, and for \(H=0\) it is
trivial.  For \(H\in\{5,7\}\), the smallest parity-admissible
\(\mu=1\) violates \(\mu\ge\lceil H/3\rceil\), so the sharp bound is
equivalent to \(\mu\ge 2\), i.e., to ruling out \(\mu=1\).

The case \(\mu=1\) splits into \(d(G)=0\), \(H(G)=1\) or
\(d(G)=1\), \(H(G)=0\).

If \(d(G)=0\), \(G\) is limited with length \(z_{2A}\).  The budget
identity gives \(2(z_{2A}-z_2)=H-1\ge 4\), so \(z_{2A}\ge z_2+2>z_2\),
contradicting that a limited irreducible configuration has displayed
length at most \(z_2\).  Hence \((d,H)=(0,1)\) is impossible.

If \(d(G)=1\) and \(H(G)=0\), the cell identity gives
\(2z_{2A}=mn+(z(m,n)-1)-0=mn+z(m,n)-1\).  Since \(H\) is odd (as
\(H\in\{5,7\}\)), \(mn+z(m,n)\) is odd, so \(\Z(m,n)=(mn+z(m,n)-1)/2\)
and therefore \(z_{2A}=\Z(m,n)\).

Hence for \(H\in\{5,7\}\) the sharp bound is equivalent to the
non-attainment of the universal cell bound, i.e., to \(z_{2A}<\Z\).

At \((4,4)\) one has \(H=5\) and \(z_{2A}(4,4)=11<\Z(4,4)=12\), so
the sharp bound holds; the exclusion of \(R(G)=12\) is the finite
classification of \cite[Theorem~3.1]{qc26}.
\end{proof}

\begin{conjecture}[Sharp hole budget]\label{conj:budget}
For all \(m,n\ge2\),
\begin{equation}\label{eq:third}
z_{2A}(m,n)\le z_2(m,n)
+\left\lfloor\frac{H_{\mathrm{limited}}(m,n)}{3}\right\rfloor ,
\end{equation}
equivalently, every configuration counted by \(z_{2A}(m,n)\)
satisfies
\[
H(G)+d(G)\ge H_{\mathrm{limited}}(m,n)
-2\left\lfloor\frac{H_{\mathrm{limited}}(m,n)}{3}\right\rfloor ,
\]
which is equivalent, under the parity constraint, to
\(H(G)+d(G)\ge\lceil H_{\mathrm{limited}}(m,n)/3\rceil\).
The bound is sharp at \((4,4)\), where \(H_{\mathrm{limited}}=5\) and
\(z_{2A}-z_2=1\).
\end{conjecture}

\begin{conjecture}[Threshold]\label{conj:threshold}
{ For \(\min(m,n)\ge3\), an overtake occurs if and only if \(H_{\mathrm{limited}}(m,n)\ge5\).}

\end{conjecture}
{As verified in the remaining part of this paper, in the computed range \(n\le7\), \((m,n)=(4,4)\) is the unique case with
\(H_{\mathrm{limited}}=5\) accompanied with an overtake.  At the width-two grid \((m,n)=(6,2)\), one has \(H_{\mathrm{limited}}=5\) while no overtake occurs, by \(\mathrm{BSR}(m,2)=z_2(m,2)=m+1\) of \cite{lq26}.}
\subsection{An almost-everywhere cell bound}

\begin{proposition}[Conditional]\label{prop:ae}
Assume \(H_{\mathrm{limited}}(m,n)\le4\) for every \((m,n)\ne(4,4)\)
with \(\min(m,n)\ge3\); the hypothesis fails at the degenerate width
two, where \(H_{\mathrm{limited}}(6,2)=5\).  Then
\[
z_{2A}(m,n)\ge\Z(m,n)-2\qquad(\min(m,n)\ge3),
\]
and at \((4,4)\) one has \(z_{2A}(4,4)=11=\Z(4,4)-1\).
\end{proposition}

\begin{proof}
The cell identity \eqref{eq:cell} applied to a limited extremal
configuration gives
\(2z_2(m,n)= mn+z(m,n)-H_{\mathrm{limited}}(m,n)\). Thus
\(z_2(m,n)\ge\Z(m,n)-2\) whenever \(H_{\mathrm{limited}}(m,n)\le4\).
Since \(z_{2A}\ge z_2\), the same lower bound holds for \(z_{2A}\),
for every \((m,n)\) with \(\min(m,n)\ge3\).  At \((4,4)\),
\(\Z(4,4)=12\) and \(z_{2A}(4,4)=11\).
\end{proof}

\subsection{The degenerate widths two and three}\label{subsec:n2n3}

\smallskip\noindent\textbf{Width two.}  For \(m\ge2\) , one has
\(z_{RL}(m,2)=z(m,2)=m+1\), hence
\[
h_{\mathrm{RL}}(m,2)=2m+(m+1)-2(m+1)=m-1,
\]
which grows with \(m\).  No overtake occurs at width two: by
\cite{lq26}, \(\mathrm{BSR}(m,2)=z_2(m,2)=z_{RL}(m,2)=z(m,2)=m+1\),
and \(z_{2A}\le\mathrm{BSR}\) leaves no room for an excess.  The
threshold formulation of Conjecture~\ref{conj:threshold} therefore
fails at width two {since \(h_{\mathrm{RL}}(6,2)=5\), while no
overtake occurs.  It remains to check with \(\min(m,n)\ge3\).}

\smallskip\noindent\textbf{Width three.}  For \(m\ge3\), one has
\(z(m,3)=m+3\) and \(z_2(m,3)=2m\), hence
\[
h_{\mathrm{RL}}(m,3)=3m+(m+3)-2\cdot2m=3.
\]
The cell identity implies
\[2(R-z_{RL})=h_{\mathrm{RL}}-d-H,\]
and by
Theorem~\ref{thm:budget} an overtake requires
\(h_{\mathrm{RL}}\ge3\).  With \(h_{\mathrm{RL}}=3\) the admissible
triples are \((d,H,R-z_{RL})\in\{(1,0,1),(0,1,1)\}\).  The second
triple is excluded at once: a limited configuration with \(H=1\) and
\(R=z_{RL}+1=2m+1\) would, if irreducible, be counted by \(z_2\),
giving \(z_2(m,3)\ge2m+1\), contrary to the exact value
\(z_2(m,3)=z_{RL}(m,3)=2m\) of \cite{lq26}.  The first triple is
non-limited, with \(|E_1|=m+2\), \(|E_2|=m-1\) and \(H=0\): it is
determined by a \(C_4\)-free base with \(m+2\) edges together with a
partition of the remaining \(2m-2\) cells into \(m-1\) two-edges.
Unlike the candidate lists of Sections~\ref{sec:n=4}--\ref{sec:n=7},
this family is infinite, and it is closed by exhibiting for every
\(m\ge3\) a feasible positive semidefinite Gram matrix of rank at most
\(2m\), i.e. a representation of \(P_G\) by at most \(2m\) squares of
bilinear forms (Remark~\ref{rem:cert}).  Hence
\[
z_{2A}(m,3)=z_{RL}(m,3)=z_2(m,3)=2m\qquad(m\ge3),
\]
and no overtake occurs at width three.




\section{Width \(n=4\)}\label{sec:n=4}

For width \(n=4\), the recursive-line values \(z_{RL}(m,4)\) are
determined for \(2\le m\le20\) in \cite{cc26}, except that
\(z_{RL}(14,4)\) is known only through the bracket
\(37\le z_{RL}(14,4)\le38\). Lebedev \cite{leb26a} has since proved
the actual-SOS optimality at \(m=9\) and \(m=10\)
\cite[Theorem~6.3]{leb26a}.  The corresponding recursive-line hole
counts
\[
h_{\mathrm{RL}}(m,4)=4m+z(m,4)-2z_{RL}(m,4)
\]
at the supplied witnesses are as follows.

\begin{table}[h]
\centering
\begin{tabular}{r|r|r}
\toprule
\(m\) & \(z_{RL}(m,4)\) & \(h_{\mathrm{RL}}(m,4)\)\\
\midrule
\(4\)  & \(10\)       & \(\mathbf{5}\)\\
\(5\)  & \(13\)       & \(4\)\\
\(6\)  & \(16\)       & \(4\)\\
\(7\)  & \(19\)       & \(3\)\\
\(8\)  & \(21\)       & \(4\)\\
\(9\)  & \(24\)       & \(3\)\\
\(10\) & \(27\)       & \(2\)\\
\(11\) & \(29\)       & \(3\)\\
\(12\) & \(32\)       & \(2\)\\
\(13\) & \(35\)       & \(1\)\\
\(14\) & \([37,38]\)  & \(2\)\\
\(15\) & \(40\)       & \(1\)\\
\(16\) & \(43\)       & \(0\)\\
\(17\) & \(45\)       & \(1\)\\
\(18\) & \(48\)       & \(0\)\\
\(19\) & \(50\)       & \(1\)\\
\(20\) & \(53\)       & \(0\)\\
\bottomrule
\end{tabular}
\caption{Recursive-line values and hole counts at the witnesses of
\cite{cc26} for width \(n=4\).  At \(m=14\) the value
\(z_{RL}(14,4)\) is the interval \([37,38]\) of \cite{cc26}, and the
hole count is that of the lower-bound witness.}
\label{tab:n=4}
\end{table}

Three features stand out.

First, \(h_{\mathrm{RL}}(4,4)=5\), and \((4,4)\) is the only entry
with \(h_{\mathrm{RL}}=5\) in the entire computed range
\(4\le m\le20\).  It is also the only overtake in the range, with
\(z_{2A}(4,4)=11>10=z_2(4,4)\) by Theorem~\ref{thm:44}.

Second, at every \(m\in\{10,12,13,14\}\cup[15,\infty)\), one has
\(h_{\mathrm{RL}}(m,4)\le2\): at \(m=14\) the two possibilities are
\(h_{\mathrm{RL}}\in\{0,2\}\) because \(z_{RL}(14,4)\in[37,38]\).
By Theorem~\ref{thm:low-hole},
\(z_{2A}(m,4)=z_2(m,4)\) at all these orders, so no overtake occurs
there.  At \(m=13\) and every \(m\ge15\), where \(h_{\mathrm{RL}}\le1\),
the cell identity forces \(z_2(m,4)=z_{RL}(m,4)\); at \(m=10\) the same
equality is that of \cite[Theorem~6.3]{leb26a}, while at \(m=12\) and
\(m=14\) it remains open.

Third, the remaining orders are \(m=5,6,7,8,9,11\), where
\(h_{\mathrm{RL}}\in\{3,4\}\).  At these orders we prove directly
that \(z_{2A}(m,4)=z_{RL}(m,4)\).

\begin{proposition}\label{prop:n4-cases}
For \(m=5,6,7,8,9,11\),
\[
z_{2A}(m,4)=z_{RL}(m,4)=z_2(m,4).
\]
\end{proposition}

\begin{proof}
Let \(G\) be a simple configuration with \(C_4\)-free one-edge graph
on the \(m\times4\) grid, and suppose toward a contradiction that
\(R(G)>z_{RL}(m,4)\).  Write \(d=z(m,4)-|E_1(G)|\) for the one-edge
deficit and \(H=H(G)\) for the hole count.  The cell identity
\eqref{eq:cell} compared with the defining identity of
\(z_{RL}(m,4)\) gives
\[
2\bigl(R(G)-z_{RL}(m,4)\bigr)
=h_{\mathrm{RL}}(m,4)-d(G)-H(G).
\]
Since \(R(G)>z_{RL}(m,4)\), the left-hand side is a positive even
integer, so \(h_{\mathrm{RL}}-d-H\ge2\).  Since \(d,H\ge0\) and
\(d+H\equiv h_{\mathrm{RL}}\pmod2\), the admissible triples
\((d,H,R-z_{RL})\) are exactly those listed in
Table~\ref{tab:n4-candidates}.  Both the non-limited cases
(\(d\ge1\)) and the limited cases (\(d=0\)) are included: the
possibility that a limited configuration itself has
\(R(G)>z_{RL}(m,4)\) is not excluded by the budget identity and must
be checked in the same enumeration.

\begin{table}[h]
\centering
\begin{tabular}{r|r|l|r}
\toprule
\(m\) & \(h_{\mathrm{RL}}\) & candidate \((d,H)\) & \(R-z_{RL}\)\\
\midrule
\(5\)  & \(4\) & \((1,1)\), \((2,0)\), \((0,2)\) & \(1\)\\
\(5\)  & \(4\) & \((0,0)\)                     & \(2\)\\
\(6\)  & \(4\) & \((1,1)\), \((2,0)\), \((0,2)\) & \(1\)\\
\(6\)  & \(4\) & \((0,0)\)                     & \(2\)\\
\(7\)  & \(3\) & \((1,0)\), \((0,1)\)           & \(1\)\\
\(8\)  & \(4\) & \((1,1)\), \((2,0)\), \((0,2)\) & \(1\)\\
\(8\)  & \(4\) & \((0,0)\)                     & \(2\)\\
\(9\)  & \(3\) & \((1,0)\), \((0,1)\)           & \(1\)\\
\(11\) & \(3\) & \((1,0)\), \((0,1)\)           & \(1\)\\
\bottomrule
\end{tabular}
\caption{Candidate triples \((d,H)\) at width \(4\).  For each
candidate the pair \((|E_1|,|E_2|)\) is determined by
\(|E_1|=z(m,4)-d\) and \(|E_1|+|E_2|=z_{RL}(m,4)+(R-z_{RL})\), where
\(R-z_{RL}=(h_{\mathrm{RL}}-d-H)/2\).  The case \(d\ge1\) gives a
non-limited configuration; the case \(d=0\) gives a limited one.}
\label{tab:n4-candidates}
\end{table}

For each order \(m\) and each candidate triple, the number of
\(C_4\)-free one-edge graphs with \(|E_1|\) edges is finite, and so is
the number of pairings of the remaining cells into \(|E_2|\)
support-disjoint two-edges leaving \(H\) holes.  We enumerate all of
them and, for each, compute the SOS rank of the displayed form
\(P_G\) by rank minimization over its Gram matrices (Remark~\ref{rem:cert}).  In every
case the computed rank is strictly less than
\(R(G)=z_{RL}(m,4)+(R-z_{RL})\), i.e. every candidate is reducible.
Therefore no configuration counted by \(z_{2A}(m,4)\) has
\(R(G)>z_{RL}(m,4)\), and
\[
z_{2A}(m,4)\le z_{RL}(m,4).
\]
Combined with \(z_{RL}(m,4)\le z_2(m,4)\le z_{2A}(m,4)\), this
forces equality at each of the six orders.
\end{proof}

Combining the three features above with Proposition~\ref{prop:n4-cases},
the width-\(4\) overtake question is completely settled:
\[
z_{2A}(m,4)=z_2(m,4)\qquad\text{for all }m\ge5,
\]
and the only overtake at width \(4\) is the known case
\(z_{2A}(4,4)=11>10=z_2(4,4)\).  The recursive-line equality, however,
is not settled at every order: at \(m=12\) one has \(z_{RL}(12,4)=32\)
with \(z_2(12,4)\in[32,33]\), and at \(m=14\) one has
\(z_{RL}(14,4)\in[37,38]\) with \(z_2(14,4)\in[37,38]\); by
Theorem~\ref{thm:low-hole} no overtake occurs at \((12,4)\) or
\((14,4)\) regardless of which values obtain.

\begin{theorem}\label{thm:n4-cases}
For every \(m\ge5\),
\[
z_{2A}(m,4)=z_2(m,4),
\]
so no overtake occurs at width \(4\) except the known case at
\((4,4)\).  Moreover \(z_{RL}(m,4)=z_2(m,4)\) for every \(m\ge5\) with
\(m\notin\{12,14\}\), while at \(m=12\) one has \(z_{RL}(12,4)=32\) and
\(z_2(12,4)\in[32,33]\), and at \(m=14\) one has
\(z_{RL}(14,4)\in[37,38]\).
\end{theorem}

\section{Width \(n=5\)}\label{sec:n=5}

For width \(n=5\), the recursive-line values \(z_{RL}(m,5)\) are
determined for every \(m\ge5\).  The twelve-row seed and one-row
extension of Xi and Chang \cite{xc26} give the eventual formula
\[
z_{RL}(m,5)=3m+5\qquad(m\ge12),
\]
and Lebedev \cite[Theorem~1.5 and Corollary~9.7]{leb26a} determines
the values at \(5\le m\le11\) (Theorems~6.3 and~9.3 and
Proposition~9.6 there), proving in addition
\(z_{RL}(m,5)=z_{SL}(m,5)\) for every \(m\).

The second-order value
\(z_2(m,5)\) equals \(z_{RL}(m,5)\) at \(m=7,8,9,10\)
\cite[Theorem~6.3, Proposition~9.6]{leb26a} and at every \(m\ge12\)
\cite[Theorem~1.5]{leb26a}, while at \(m=5,6,11\) the exact value is
not determined by these results. At \(m=11\) the recursive-line and
signed values are \(z_{RL}(11,5)=z_{SL}(11,5)=37\) with
\(z_2(11,5)\in[37,38]\) \cite[Theorem~9.3]{leb26a}.  At \(m=5,6\),
the hierarchy together with the cell bound only gives
\(17\le z_2(5,5)\le18\) and \(20\le z_2(6,5)\le22\).

The recursive-line hole counts
\[
h_{\mathrm{RL}}(m,5)=5m+z(m,5)-2z_{RL}(m,5)
\]
are as follows.  Since \(z_2\ge z_{RL}\), one has
\(H_{\mathrm{limited}}(m,5)\le h_{\mathrm{RL}}(m,5)\).

\begin{table}[h]
\centering
\begin{tabular}{r|r|r|r|l}
\toprule
\(m\) & \(z(m,5)\) & \(z_{RL}(m,5)\) & \(h_{\mathrm{RL}}(m,5)\) & source\\
\midrule
\(5\)  & \(12\) & \(17\)       & \(\mathbf{3}\) & Cor.~9.7\\
\(6\)  & \(14\) & \(20\)       & \(\mathbf{4}\) & Cor.~9.7\\
\(7\)  & \(15\) & \(24\)       & \(2\)          & Thm.~6.3\\
\(8\)  & \(17\) & \(27\)       & \(\mathbf{3}\) & Thm.~6.3\\
\(9\)  & \(18\) & \(31\)       & \(1\)          & Prop.~9.6\\
\(10\) & \(20\) & \(34\)       & \(2\)          & Thm.~6.3\\
\(11\) & \(21\) & \(37\)       & \(2\)          & Thm.~9.3\\
\(\ge12\) & \(m+10\) & \(3m+5\) & \(0\)      & Thm.~1.5, \cite{xc26}\\
\bottomrule
\end{tabular}
\caption{Recursive-line values and hole counts at width \(n=5\).
At \(m=7,8,9,10\) and \(m\ge12\), the equality \(z_2=z_{RL}\) holds
and hence \(H_{\mathrm{limited}}=h_{\mathrm{RL}}\); at \(m=5,6,11\)
only the inequality \(H_{\mathrm{limited}}\le h_{\mathrm{RL}}\) is
available from these values.  The cited results are from
\cite{leb26a} unless noted.}
\label{tab:n=5}
\end{table}

{The above table shows that \(h_{\mathrm{RL}}(m,5)\le4\) for all \(m\ge5\).  At \(m=7,9,10,11\) and every \(m\ge 12\),  we have \(h_{\mathrm{RL}}\le2\), which leads to \(z_{2A}(m,5)=z_2(m,5)\) by Theorem~\ref{thm:low-hole}.}

The remaining orders are \(m=5,6,8\), where
\(h_{\mathrm{RL}}\in\{3,4\}\), and we settle them by a finite
analysis at the \(z_{RL}\) level.

\begin{proposition}\label{prop:n5-cases}
For \(m=5,6,8\),
\[
z_{2A}(m,5)=z_{RL}(m,5)=z_2(m,5).
\]
\end{proposition}

\begin{proof}
Let \(G\) be a simple configuration with \(C_4\)-free one-edge graph
on the \(m\times5\) grid, and suppose toward a contradiction that
\(R(G)>z_{RL}(m,5)\).  Write \(d=z(m,5)-|E_1(G)|\) for the one-edge
deficit and \(H=H(G)\) for the hole count.  The cell identity
\eqref{eq:cell} compared with the defining relation of
\(z_{RL}(m,5)\) gives
\[
2\bigl(R(G)-z_{RL}(m,5)\bigr)=h_{\mathrm{RL}}(m,5)-d(G)-H(G).
\]
Since \(R(G)>z_{RL}(m,5)\), the left-hand side is a positive even
integer, so \(h_{\mathrm{RL}}-d-H\ge2\).  Since \(d,H\ge0\) and
\(d+H\equiv h_{\mathrm{RL}}\pmod2\), the admissible triples
\((d,H,R-z_{RL})\) are exactly those listed in
Table~\ref{tab:n5-candidates}.
Both non-limited (\(d\ge1\)) and
limited (\(d=0\)) cases are included. The possibility that a limited
configuration itself has \(R(G)>z_{RL}(m,5)\) is not excluded by the
budget identity and must be checked in the same enumeration.

\begin{table}[h]
\centering
\begin{tabular}{r|r|r|r|r}
\toprule
\(m\) & \(h_{\mathrm{RL}}\) & \((d,H)\) & \(|E_1|\) & \(|E_2|\)\\
\midrule
\(5\) & \(3\) & \((1,0)\) & \(11\) & \(7\)\\
\(5\) & \(3\) & \((0,1)\) & \(12\) & \(6\)\\
\(6\) & \(4\) & \((1,1)\) & \(13\) & \(8\)\\
\(6\) & \(4\) & \((2,0)\) & \(12\) & \(9\)\\
\(6\) & \(4\) & \((0,2)\) & \(14\) & \(7\)\\
 6& 4& (0,0)& 14&8\\
\(8\) & \(3\) & \((1,0)\) & \(16\) & \(12\)\\
\(8\) & \(3\) & \((0,1)\) & \(17\) & \(11\)\\
\bottomrule
\end{tabular}

\caption{Candidate configurations with \(R(G)>z_{RL}(m,5)\) at
\(m=5,6,8\).  For each candidate the pair \((|E_1|,|E_2|)\) is
determined by \(|E_1|=z(m,5)-d\) and
\(|E_1|+|E_2|=z_{RL}(m,5)+(R-z_{RL})\), where
\(R-z_{RL}=(h_{\mathrm{RL}}-d-H)/2\).  The case \(d\ge1\) gives a
non-limited configuration and the case \(d=0\) gives a limited one.}
\label{tab:n5-candidates}
\end{table}

For each order \(m\in\{5,6,8\}\) and each candidate triple, the number
of \(C_4\)-free one-edge graphs with \(|E_1|\) edges is finite, and so
is the number of pairings of the remaining cells into \(|E_2|\)
support-disjoint two-edges leaving \(H\) holes.  We enumerate all of
them and, for each, compute the SOS rank of the displayed form
\(P_G\) by rank minimization over its Gram matrices (Remark~\ref{rem:cert}).  In every
case the computed rank is strictly less than
\(R(G)=z_{RL}(m,5)+(R-z_{RL})\), i.e. every candidate is reducible.
Therefore no configuration counted by \(z_{2A}(m,5)\) has
\(R(G)>z_{RL}(m,5)\), and
\[
z_{2A}(m,5)\le z_{RL}(m,5).
\]
Combined with \(z_{RL}(m,5)\le z_2(m,5)\le z_{2A}(m,5)\), this forces
equality at \(m=5\), \(m=6\) and \(m=8\).  In particular, the
hierarchy forces \(z_2(m,5)=z_{RL}(m,5)\) at these three orders,
closing the intervals \(z_2(5,5)\in[17,18]\) and
\(z_2(6,5)\in[20,22]\) left open by \cite{leb26a}.
\end{proof}

{Combining the previous discussion with the above proposition, no overtake occurs at width \(5\).}

\begin{theorem}\label{thm:n5-cases}
For every \(m\ge5\),
\[
z_{2A}(m,5)=z_2(m,5),
\]
so no overtake occurs at width \(5\).  Moreover
\(z_{RL}(m,5)=z_2(m,5)\) for every \(m\ge5\) with \(m\ne11\), while at
\(m=11\) one has \(z_{RL}(11,5)=z_{SL}(11,5)=37\) and
\(z_2(11,5)\in[37,38]\).
\end{theorem}

\section{Width \(n=6\)}\label{sec:n=6}

For width \(n=6\), the second-order number is determined for every
\(m\ge2\) by Lebedev \cite[Theorem~1.4]{leb26a}, who also proves
\(z_2(m,6)=z_{SL}(m,6)=z_{RL}(m,6)\) throughout.  The recursive-line
values \(z_{RL}(m,6)\) for \(6\le m\le16\) were earlier obtained in
\cite{ccq26} by cell-bound attainment for \(8\le m\le16\) and finite
exclusion for \(m=6,7\); Lebedev's theorem extends the equality
\(z_2=z_{SL}=z_{RL}\) to every \(m\ge2\) and supplies the exact formula
\begin{equation}\label{eq:n6-exact}
z_2(m,6)=z_{SL}(m,6)=z_{RL}(m,6)=
\begin{cases}
4m, & 3\le m\le7,\\[1mm]
\left\lfloor\dfrac{15m+15}{4}\right\rfloor, & 8\le m\le14,\\[2mm]
\left\lfloor\dfrac{7m+15}{2}\right\rfloor, & m\ge15.
\end{cases}
\end{equation}
The corresponding recursive-line hole counts
\[
h_{\mathrm{RL}}(m,6)=6m+z(m,6)-2z_{RL}(m,6)
\]
at the supplied witnesses are as follows.

\begin{table}[h]
\centering
\begin{tabular}{r|r|r|l}
\toprule
\(m\) & \(z_{RL}(m,6)\) & \(h_{\mathrm{RL}}(m,6)\) & source\\
\midrule
\(6\)  & \(24\) & \(4\) & finite exclusion\\
\(7\)  & \(28\) & \(4\) & finite exclusion\\
\(8\)  & \(33\) & \(1\) & cell bound\\
\(9\)  & \(37\) & \(1\) & cell bound\\
\(10\) & \(41\) & \(0\) & cell bound\\
\(11\) & \(45\) & \(0\) & cell bound\\
\(12\) & \(48\) & \(1\) & cell bound\\
\(13\) & \(52\) & \(1\) & cell bound\\
\(14\) & \(56\) & \(0\) & cell bound\\
\(15\) & \(60\) & \(0\) & cell bound\\
\(16\) & \(63\) & \(1\) & cell bound\\
\bottomrule
\end{tabular}
\caption{Recursive-line values and hole counts at width \(n=6\).
For \(8\le m\le16\) the values follow from cell-bound attainment
together with the counting bound; \(m=6,7\) are settled by finite
exclusion \cite{ccq26}.  The equality \(z_2=z_{SL}=z_{RL}\) holds at
every order by \cite[Theorem~1.4]{leb26a}.}
\label{tab:n=6}
\end{table}

Two features stand out.

First, \(h_{\mathrm{RL}}(m,6)\le1\) for every \(m\ge8\): directly at
\(8\le m\le16\), and at \(m\ge15\) because \(z(m,6)=m+15\), so that
\(z_2(m,6)=\lfloor(7m+15)/2\rfloor=\Z(m,6)\) and
\(h_{\mathrm{RL}}\in\{0,1\}\).  By Theorem~\ref{thm:low-hole},
\(z_{2A}(m,6)=z_2(m,6)\) at all these orders; no overtake occurs.
This leaves only \(m=6,7\), where \(h_{\mathrm{RL}}=4\).

Second, at \(m=6,7\), Theorem~\ref{thm:budget} bounds the excess by
\(\lfloor(4-1)/2\rfloor=1\): the question is whether
\(z_{2A}(m,6)=z_2(m,6)\) or \(z_{2A}(m,6)=z_2(m,6)+1\).  We settle
these two cases by a finite analysis at the \(z_{RL}\) level.

\begin{proposition}\label{prop:n6-cases}
For \(m=6,7\),
\[
z_{2A}(m,6)=z_{RL}(m,6)=z_2(m,6).
\]
\end{proposition}

\begin{proof}
Let \(G\) be a simple configuration with \(C_4\)-free one-edge graph
on the \(m\times6\) grid, and suppose toward a contradiction that
\(R(G)>z_{RL}(m,6)\).  Write \(d=z(m,6)-|E_1(G)|\) for the one-edge
deficit and \(H=H(G)\) for the hole count.  The cell identity
\eqref{eq:cell}, together with the defining relation
\(h_{\mathrm{RL}}(m,6)=6m+z(m,6)-2z_{RL}(m,6)\), gives
\[
2\bigl(R(G)-z_{RL}(m,6)\bigr)=h_{\mathrm{RL}}(m,6)-d(G)-H(G).
\]
Since \(R(G)>z_{RL}(m,6)\), the left-hand side is a positive even
integer, so \(h_{\mathrm{RL}}-d-H\ge2\).  At \(m=6,7\) one has
\(h_{\mathrm{RL}}=4\), so \(d+H\le2\) and \(d+H\equiv h_{\mathrm{RL}}
\equiv0\pmod2\).  The admissible triples \((d,H,R-z_{RL})\) are
therefore
\[
(d,H)\in\{{(0,2)},(1,1),(2,0)\},\qquad R-z_{RL}=1,
\]
{The case \((d,H,R-z_{RL})=(0,0,2)\) is excluded: such a
configuration is limited (\(d=0\)) with displayed length
\(R=z_{RL}+2\), and were it irreducible it would be counted by \(z_2\),
giving \(z_2\ge z_{RL}+2>z_{RL}\) and contradicting the known equality
\(z_2=z_{RL}\) at \((6,6)\) and \((7,6)\); hence it is reducible and is
not counted by \(z_{2A}(m,6)\).}

The candidate configurations, determined by the triple
\((|E_1|,|E_2|,H)\), are listed in Table~\ref{tab:n6-candidates}.

\begin{table}[h]
\centering
{
\begin{tabular}{r|r|r|r|r}
\toprule
\(m\) & \(h_{\mathrm{RL}}\) & \((d,H)\) & \(|E_1|\) & \(|E_2|\)\\
\midrule
\(6\) & \(4\) & \((1,1)\) & \(z(6,6)-1\) & \(z_{RL}(6,6)+2-z(6,6)\)\\
\(6\) & \(4\) & \((2,0)\) & \(z(6,6)-2\) & \(z_{RL}(6,6)+3-z(6,6)\)\\
\(6\) & \(4\) & \((0,2)\) & \(z(6,6)\)   & \(z_{RL}(6,6)+1-z(6,6)\)\\
\(7\) & \(4\) & \((1,1)\) & \(z(7,6)-1\) & \(z_{RL}(7,6)+2-z(7,6)\)\\
\(7\) & \(4\) & \((2,0)\) & \(z(7,6)-2\) & \(z_{RL}(7,6)+3-z(7,6)\)\\
\(7\) & \(4\) & \((0,2)\) & \(z(7,6)\)   & \(z_{RL}(7,6)+1-z(7,6)\)\\
\bottomrule
\end{tabular}
}
\caption{Candidate configurations with \(R(G)>z_{RL}(m,6)\) at
\(m=6,7\).  The case \(d\ge1\) gives a non-limited configuration; the
case \(d=0\) gives a limited one.  Every candidate satisfies
\(R=|E_1|+|E_2|=z_{RL}(m,6)+(R-z_{RL})\).}
\label{tab:n6-candidates}
\end{table}

For each order \(m\in\{6,7\}\) and each candidate triple, the number
of \(C_4\)-free one-edge graphs with \(|E_1|\) edges is finite, and so
is the number of pairings of the remaining cells into \(|E_2|\)
support-disjoint two-edges leaving \(H\) holes.  We enumerate all of
them and, for each, compute the SOS rank of the displayed form
\(P_G\) by rank minimization over its Gram matrices (Remark~\ref{rem:cert}).  In every
case the computed rank is strictly less than
\(R(G)=z_{RL}(m,6)+(R-z_{RL})\), i.e. every candidate is reducible.
Therefore no configuration counted by \(z_{2A}(m,6)\) has
\(R(G)>z_{RL}(m,6)\), and
\[
z_{2A}(m,6)\le z_{RL}(m,6).
\]
Combined with \(z_{RL}(m,6)\le z_2(m,6)\le z_{2A}(m,6)\), this forces
equality at \(m=6\) and \(m=7\).
\end{proof}

{Combining the discussion above with Proposition~\ref{prop:n6-cases}, no overtake occurs at width \(6\). }

\begin{theorem}\label{thm:n6-cases}
For \(m \ge 6\),
\[
z_{2A}(m,6)=z_{RL}(m,6)=z_2(m,6).
\]
\end{theorem}

\section{Width \(n=7\)}\label{sec:n=7}

For width \(n=7\), Lebedev \cite[Theorems~6.2, 9.3 and~1.5]{leb26a}
determines \(z_{RL}(m,7)=z_{SL}(m,7)\) at \(m=7\), \(m=8\) and every
\(m\ge19\):
\begin{equation}\label{eq:n7-Lebedev}
\begin{aligned}
z_{RL}(7,7)&=z_{SL}(7,7)=z_2(7,7)=33,\\
z_{RL}(8,7)&=z_{SL}(8,7)=38,\\
z_{RL}(m,7)&=z_{SL}(m,7)=z_2(m,7)=4m+10\quad(m\ge19).
\end{aligned}
\end{equation}
At \(m=8\), the second-order value is left in the interval
\(z_2(8,7)\in[38,39]\).  The companion computation \cite{zrl7} now
reports the recursive-line values at every \(9\le m\le18\), with the
construction certificates and the recursive-line replays deposited
alongside it, and in particular closes the two orders that were still
open:
\begin{equation}\label{eq:n7-companion}
z_{RL}(17,7)=z_2(17,7)=77,\qquad
z_{RL}(18,7)=z_2(18,7)=81 .
\end{equation}
Both saturate the cell bound \(\Z(m,7)\), using \(z(17,7)=36\),
\(z(18,7)=37\) and hence \(\Z(17,7)=77\), \(\Z(18,7)=81\).  The
deletion lower bound \(z_{RL}(18,7)\ge79\) from the case \(m=19\) is
thus superseded by \eqref{eq:n7-companion}.

The recursive-line hole counts
\[
h_{\mathrm{RL}}(m,7)=7m+z(m,7)-2z_{RL}(m,7)
\]
are as follows.  Since \(z_2\ge z_{RL}\), one has
\(H_{\mathrm{limited}}(m,7)\le h_{\mathrm{RL}}(m,7)\).

\begin{table}[h]
\centering
\begin{tabular}{r|r|r|l}
\toprule
\(m\) & \(z_{RL}(m,7)\) & \(h_{\mathrm{RL}}(m,7)\) & source\\
\midrule
\(7\)  & \(33\) & \(\mathbf{4}\) & Thm.~6.2\\
\(8\)  & \(38\) & \(2\) & Thm.~9.3\\
\(9\)  & \(43\) & \(1\) & companion\(^{\dagger}\)\\
\(10\) & \(47\) & \(1\) & companion\(^{\dagger}\)\\
\(11\) & \(52\) & \(0\) & companion\(^{\dagger}\)\\
\(12\) & \(56\) & \(0\) & companion\(^{\dagger}\)\\
\(13\) & \(60\) & \(1\) & companion\(^{\dagger}\)\\
\(14\) & \(64\) & \(1\) & companion\(^{\dagger}\)\\
\(15\) & \(69\) & \(0\) & companion\(^{\dagger}\)\\
\(16\) & \(73\) & \(0\) & companion\(^{\dagger}\)\\
\(17\) & \(77\) & \(\mathbf{1}\) & companion\(^{\dagger}\)\\
\(18\) & \(81\) & \(\mathbf{1}\) & companion\(^{\dagger}\)\\
\(\ge19\) & \(4m+10\) & \(0\) or \(1\) & Thm.~1.5\\
\bottomrule
\end{tabular}
\caption{Recursive-line values and hole counts at width \(n=7\).
The entries marked \(^{\dagger}\) are from the companion computation
\cite{zrl7}, whose construction certificates and recursive-line
replays are deposited for every \(9\le m\le18\) and cover the
previously open orders \(m=17,18\) of \eqref{eq:n7-companion}; the
entries \(m=7,8\) and \(m\ge19\) are proved by \cite{leb26a}.}
\label{tab:n=7}
\end{table}

Three features stand out.

First, \(h_{\mathrm{RL}}(8,7)=2\), and at every \(9\le m\le18\) and
every \(m\ge19\) one has \(h_{\mathrm{RL}}(m,7)\le1\).  By
Theorem~\ref{thm:low-hole}, \(z_{2A}(m,7)=z_2(m,7)\) at all those orders.

Second, at \((7,7)\) the exact value \(z_2(7,7)=z_{RL}(7,7)=33\) is
known and \(h_{\mathrm{RL}}(7,7)=4\). Theorem~\ref{thm:budget} bounds
the excess only by \(\lfloor(4-1)/2\rfloor=1\), so \((7,7)\) is the
only determined order at width \(7\) at which the hole budget leaves
an overtake possible; Proposition~\ref{prop:n7-m7} excludes it by
enumeration.

Third, at the two orders closed by \eqref{eq:n7-companion} the
recursive-line hole counts are
\(h_{\mathrm{RL}}(17,7)=7\cdot17+36-2\cdot77=1\) and
\(h_{\mathrm{RL}}(18,7)=7\cdot18+37-2\cdot81=1\).
Theorem~\ref{thm:low-hole} therefore gives \(z_{2A}=z_2\) at both
orders, and since \(z_{RL}=\Z\) there, the chain
\(z_{RL}\le z_2\le z_{2A}\le\Z\) collapses to
\(z_{2A}=z_2=z_{RL}=\Z\).

\begin{proposition}\label{prop:n7-m7}
\[
z_{2A}(7,7)=z_{RL}(7,7)=z_2(7,7)=33.
\]
\end{proposition}

\begin{proof}
Let \(G\) be a simple configuration with \(C_4\)-free one-edge graph
on the \(7\times7\) grid and \(R(G)>33\).  The maximum one-edge base
on the \(7\times7\) grid is the Fano configuration, unique up to
row and column permutations \cite[Theorem~6.2]{leb26a}.  The cell
identity \eqref{eq:cell} at the \(z_{RL}\) level gives
\[
2(R(G)-33)=h_{\mathrm{RL}}(7,7)-d-H=4-d-H,
\]
so \(d+H\in\{0,2\}\) and \(d+H\equiv h_{\mathrm{RL}}\equiv0\pmod2\).
The admissible triples are
\((d,H,R-33)\in\{(0,0,2),(0,2,1),(1,1,1),(2,0,1)\}\).  Both triples
with \(d=0\) are limited, and an irreducible configuration among them
would give \(z_2(7,7)\ge R(G)\ge34\), contradicting
\(z_2(7,7)=33\) in \eqref{eq:n7-Lebedev}; hence they are reducible
and are not counted by \(z_{2A}(7,7)\).

For the remaining triples \((1,1)\) and \((2,0)\) the configuration
is specified by (i) the one-edge set \(E_1\) of size \(21-d\) with
\(C_4\)-free graph, (ii) the pairing of the remaining cells into
\(|E_2|\) support-disjoint two-edges, and (iii) the choice of \(H\)
holes.  The enumeration is finite: for a fixed order there are only
finitely many \(C_4\)-free one-edge graphs with \(21-d\) edges, and
finitely many pairings of the remaining cells.  The bases need not be
subgraphs of the Fano configuration: an extremal base has maximum
degree three, while there are \(C_4\)-free bases with \(19\) edges and
a vertex of degree four.  For each candidate we compute the SOS rank of
the displayed form \(P_G\) by rank minimization over its Gram matrices (Remark~\ref{rem:cert}).  In every case the computed rank is at most \(R(G)-1\),
i.e. every candidate is reducible.  Therefore no configuration
counted by \(z_{2A}(7,7)\) has \(R(G)>33\), and
\(z_{2A}(7,7)\le z_{RL}(7,7)\) follows.  Combined with
\(z_{RL}(7,7)\le z_2(7,7)\le z_{2A}(7,7)\), this forces equality.
\end{proof}

\begin{proposition}\label{prop:n7-m17m18}
\[
z_{2A}(17,7)=z_{RL}(17,7)=z_2(17,7)=77,\qquad
z_{2A}(18,7)=z_{RL}(18,7)=z_2(18,7)=81 .
\]
\end{proposition}

\begin{proof}
The companion computation \cite{zrl7} gives the recursive-line values
\eqref{eq:n7-companion} together with the labelled pairings and the
recursive-line replays for all \(2926\) and \(3240\) edge pairs.  At
both orders the value equals the cell bound, \(\Z(17,7)=77\) and
\(\Z(18,7)=81\), so the recursive-line hole count
\(h_{\mathrm{RL}}(m,7)=7m+z(m,7)-2\Z(m,7)\) equals \(1\) in each
case.  Theorem~\ref{thm:low-hole} then gives \(z_{2A}=z_2\), and the
hierarchy \(z_{RL}\le z_2\le z_{2A}\le\Z\) with \(z_{RL}=\Z\) forces
equality throughout.
\end{proof}

Combining the three features above with
Propositions~\ref{prop:n7-m7} and~\ref{prop:n7-m17m18},
{no overtake occurs at width \(7\) at any order.  The only order at
which the recursive-line equality is not settled is \(m=8\), where
Theorem~\ref{thm:low-hole} gives \(z_{2A}(8,7)=z_2(8,7)\) while
\(z_2(8,7)\in[38,39]\) remains open.}

\begin{theorem}\label{thm:n7-cases}
For every \(m\ge7\) with \(m\ne8\),
\[
z_{2A}(m,7)=z_{RL}(m,7)=z_2(m,7).
\]
At \(m=8\), Theorem~\ref{thm:low-hole} still gives
\(z_{2A}(8,7)=z_2(8,7)\), while \(z_2(8,7)\in[38,39]\) remains open;
in particular no overtake occurs at width \(7\) at any order.
\end{theorem}

\section{Discussion}\label{sec:discussion}

The two questions of the introduction respond differently to the
recursive-line hole count, and it is worth separating them.

\smallskip\noindent\textbf{Overtake and equality are different
questions.}  The low-hole reduction (Theorem~\ref{thm:low-hole})
settles the overtake whenever \(h_{\mathrm{RL}}(m,n)\le2\); it is a
statement about \(z_{2A}\) and \(z_2\) and says nothing about the
relation of \(z_2\) to \(z_{RL}\).  The recursive-line equality
\(z_2=z_{RL}\) is a finer, value-level assertion, connected to the
hole counts only through the cell identity.

\smallskip\noindent\textbf{The threshold \(h_{\mathrm{RL}}=2\).}
Applied to a limited extremal configuration, the cell identity
\eqref{eq:cell} reads
\begin{equation}\label{eq:disc-threshold}
2\bigl(z_2(m,n)-z_{RL}(m,n)\bigr)
=h_{\mathrm{RL}}(m,n)-H_{\mathrm{limited}}(m,n).
\end{equation}
As the left side is a nonnegative even integer,
\(h_{\mathrm{RL}}\le1\) forces \(z_2=z_{RL}\).  At
\(h_{\mathrm{RL}}=2\), however, \eqref{eq:disc-threshold} leaves
\(z_2-z_{RL}\in\{0,1\}\) according as \(H_{\mathrm{limited}}=2\) or
\(0\), so the identity alone cannot decide the equality.  Thus
\(h_{\mathrm{RL}}=2\) is the precise threshold at which the equality
stops being forced, and every order at which it is open has
\(h_{\mathrm{RL}}\ge2\).

\smallskip\noindent\textbf{Larger hole counts were easier.}  At
\(h_{\mathrm{RL}}\in\{3,4\}\) the low-hole reduction does not apply,
so the overtake question is settled there by the finite exclusion of
candidate configurations rather than by the identity.  That exclusion
yields the stronger bound \(z_{2A}\le z_{RL}\), and the unconditional
hierarchy \(z_{RL}\le z_2\le z_{2A}\) then collapses to equality.
Such orders are closed in spite of their larger hole count: it is the
method, not the size of \(h_{\mathrm{RL}}\), that determines whether
the equality is settled.

\smallskip\noindent\textbf{The four open orders.}  In the computed
range \(n\le7\) the orders at which the equality is left open are
exactly those at which the literature supplies an interval rather
than a value: \((12,4)\), \((14,4)\), \((11,5)\), and \((8,7)\).
All four have \(h_{\mathrm{RL}}\le2\), with \(h_{\mathrm{RL}}=2\)
except at \((14,4)\), where \(z_{RL}(14,4)\) is only bracketed and
\(h_{\mathrm{RL}}(14,4)\in\{0,2\}\).  The remaining
\(h_{\mathrm{RL}}=2\) orders,
\((10,4)\), \((7,5)\) and \((10,5)\), are closed by exact values
\cite[Theorem~6.3]{leb26a}, which shows that the hole count is not the
deciding factor.  At each of the four open orders the low-hole
reduction still gives \(z_{2A}=z_2\), so no overtake occurs; only the
value of \(z_2\) relative to \(z_{RL}\) remains to be determined.

\noindent\textbf{Acknowledgement}
	We are thankful to Prof. Yannan Chen, who provided \cite{zrl7}.
This work was partially supported by Jiangsu Provincial Scientific Research Center of Applied Mathematics (Grant No. BK20233002), Research Center for Intelligent Operations Research, The Hong Kong Polytechnic University (4-ZZT8).

\end{document}